\documentclass[12pt]{amsart}
\usepackage{amsmath,amsthm,amssymb,enumerate,cases,graphicx,multirow,subcaption}

\usepackage[T2A]{fontenc}
\usepackage[utf8]{inputenc}
\usepackage[russian,english,showlanguages]{babel}
\usepackage{hyphenat}
\usepackage{epigraph}

\renewcommand\div{{\rm div\,}}
\newcommand\sgn{{\rm sgn\,}}

\newtheorem{thm}{Theorem}

\title{Augmented data for the Backus problem}
\author{Dmitry Glotov}
\date{}

\begin{document}

\bibliographystyle{plain}
\begin{abstract}
Backus~\cite{Backus68} considered a boundary value problem for the Laplace equation with the non-linear data in the form of the magnitude $|Du|$ of the gradient of the solution $u$. We consider this problem with the data expanded by $(\partial/\partial\nu)|Du|$ given on the boundary of the domain. To justify the requirement for additional data, we use them to estimate the number of sources for the related inverse source problem in the plane. We show that, for an arbitrary dimension, a harmonic function satisfies a quasi-linear equation on the boundary of the domain with the coefficients involving the augmented data.
We use the finite element method to recover the harmonic function on the boundary by solving numerically the derived equation. 
\end{abstract}
\maketitle

\section{Introduction}\label{sec:intro}
Backus~\cite{Backus68} considered the problem of determining a harmonic function $u$
given the magnitude $|Du|$ of the gradient.
This problem has geophysical applications in gravimetric and geomagnetic (intensity) surveys. For example in magnetometry, in the absence of currents and away from singularities, Maxwell's equation for the magnetic field reduces to
the Laplace equation. 
The intensity of a magnetic or gravitational field is more readily available than the full field data, which also includes the field direction. Other potential applications can be found in medicine, specifically in magnetoencephalography and magnetocardiography~\cite{MEG}. Traditionally in these fields, the full magnetic field is measured using superconducting quantum interface devices (SQUID). These devices are bulky and their use is associated with a high cost of acquisition and maintenance while {\em scalar} magnetometers, which measure only the magnitude of the magnetic field, are in comparison compact, scalable, and capable of achieving a high level of sensitivity.

The focus of Backus' paper is on the data in the form of $|Du|$ given on the boundary of the domain.
Backus considered the problem for the cases of a bounded domain, an unbounded domain with the bounded exterior, and the half-plane. For the exterior of a sphere in $\mathbb R^3$, Backus proved~\cite{Backus70} that the problem admits multiple solutions and additional data are required for well-posedness.
In two dimensions, the problem for bounded domains has a unique solution~\cite{Backus68} provided we are given the locations of all zeros of $Du$ in $\Omega$ and the direction of $Du$ at one point in $\Omega$.
This additional requirement results in a well-posed problem, however the data of this type may be viewed as exogenous.
Another example of expanded data is the requirement of a definite sign for the normal derivative of the solution on the boundary of an unbounded domain. In this case, Backus established uniqueness under the constraint that the solution vanishes at infinity~\cite{Backus68}, which itself can be interpreted as an additional boundary condition. Cherkaeva~\cite{Cherkaeva} considered a linearized problem for the exterior of a ball in $\mathbb R^3$. She obtained uniqueness by requiring additional data in the form of the potential or its normal derivative on the {\em magnetic equator} of the sphere. The magnetic equator in this context is defined as the curve on the sphere on which the vertical component of the main field vanishes. A related result was obtained by Jorge and Magnanini~\cite{JM93}. They proved that, if $u$ and $v$ are two harmonic functions defined outside the unit ball and regular at infinity such that $|Du|=|Dv|$ on $S^2$ and $u=v$ on the set $\{x\in S^2:\ (\partial/\partial\nu)(u+v)=0\}$, then $u=v$.
Also for the exterior problem, Backus proved that, if $|Du|=|Dv|$ throughout the domain, then $u=v$. The last result sets a precedent for assuming the knowledge of $|Du|$ and hence its derivatives on a larger set in a special formulation of the problem. 

We propose an intermediate assumption, namely, we require the data in the form of $|Du|$ and its normal derivative of $(\partial/\partial\nu)|Du|$ to be given on the boundary of the domain.  
The data in the equivalent form $(p,Dp):=(|Du|,D|Du|)$ was considered by this author in~\cite{G10} for planar domains.  In the case of two independent variables, the non-linear equations relating the data $(p,Dp)$ and the solution $(Du,D^2u)$ take the form
\begin{equation}\label{eq:nonlin2d}
\begin{aligned}
p & =u_1^2+u_2^2,\\
p_1 & =2(u_{11}u_1+u_{12}u_2),\\
p_2 & =2(u_{12}u_1+u_{22}u_2),
\end{aligned}
\end{equation}
where the subscripts denote partial derivatives, i.e., $u_j=\partial u/\partial x_j$, $u_{jk}=\partial^2u/\partial x_j\partial x_k$, etc.  
These equations are equivalent to a system of linear equations for the components of the solution $(Du,D^2u)$ with the components of the data $(p,Dp)$ appearing as coefficients:
\begin{equation}\label{lin2d}
\begin{aligned}
p_1u_1-p_2u_2-2p\,u_{11} & =0,\\
p_2u_1+p_1u_2-2p\,u_{12} & =0,
\end{aligned}
\end{equation}
This system can be reduced to a system of two ordinary differential equations and further rewritten as a single ordinary differential equation in one complex variable, which simplifies the analysis and lends itself readily to reconstruction methods including analytic solution formulas and numerical computation.

The goal of the present paper is to extend the recovery of $(u,Du)$ from $(p,Dp)$ to higher dimensions.
The organization is as follows. In Section~\ref{sec:source}, 
we consider a relevant inverse source problem and revisit an estimate from~\cite{G10} on the number of sources in terms of the expanded data in the two-dimensional case. 
In the case of a general dimension $d\ge2$, we derive in Section~\ref{sec:mielofon} a non-linear elliptic equation satisfied by a harmonic function $u$ on a hyperplane with coefficients involving $|Du|$ and $(\partial/\partial\nu)|Du|$. We report in Section~\ref{sec:numer} on numerical studies of recovering $u$ (and $Du$) on $\partial\Omega$ from the knowledge of $|Du|$ and $(\partial/\partial\nu)|Du|$ (or $D|Du|$). The method is based on solving the non-linear equation derived in Section~\ref{sec:mielofon}.

In addition to the already mentioned references~\cite{Backus68,Backus70,JM93,G10}, the Backus problem was studied by D\'{\i}az, D\'{\i}az, and Otero~\cite{DDO}. These authors considered an oblique derivative problem related to the exterior Backus problem by the Kelvin transform. In particular, they proved the maximality of the solution of the exterior problem with a definite sign of the normal derivative.  Lieberman~\cite{L87} addressed the regularity of solutions for a class of nonlinear boundary value problems that includes the Backus problem on bounded domains. Payne and Schaeffer~\cite{PS03} obtained bounds for solutions of parabolic equations and their gradient with the magnitude of the gradient prescribed on the boundary of bounded domains. Kaiser and Neudert~\cite{KN04} and Kaiser~\cite{Kaiser10} considered the exterior problem with the direction of the gradient given on the boundary and the magnitude unknown. The data in the form of the magnitude of the gradient was also studied in connection with the problem of electrical impedance tomography by Nachman, Tamasan, and Timonov~\cite{NTT7,NTT9}.  These authors considered the problem of finding the conductivity $\sigma$ from the measurement of the magnitude of the current $|J|$ where the current is given by $J=\sigma Du$ and $u$ is a solution of the elliptic equation $D\cdot(\sigma Du)=0$. In addition, either the Cauchy data~\cite{NTT7} or the Dirichlet data~\cite{NTT9} are assumed to be given on the boundary.

We point out that the collection of data in the form of the gradient of $|Du|$ for a potential $u$ is already implemented in magnetic surveys in archaeological geophysics and in gravimetric surveys with the help of devices called {\em gradiometers}. These devices consist of two or more sensors forming a basis for a finite difference approximation of the gradient of the field magnitude.

\section{Estimate for the number of poles}\label{sec:source}
The recovery of $(u,Du)$ from $(p,Dp)$ on the boundary may be viewed as a first step in the identification of points sources located inside the domain given the magnitude of the field data on the boundary. The second step of finding the potential 
assuming a specific form of the solution can be classified as an inverse {\em source} problem for the Poisson equation with point sources.  

To state this problem more precisely let $\Omega$ be a domain in $\mathbb R^d$ where $d\ge2$ and consider the equation
\begin{equation}\label{Poisson}
\Delta u=f,
\end{equation}
with the source term $f$ of the form
\begin{equation}\label{f}
f=\sum_{|\alpha|\le N}\sum_{j=1}^{N_\alpha}b^\alpha_jD^\alpha\delta_{x^\alpha_j},
\end{equation}
where $\alpha=(\alpha_1,\dots,\alpha_d)\in(\mathbb N\cup\{0\})^d$ is a multiindex with
\[|\alpha|=\alpha_1+\dots+\alpha_d\quad\text{and}\quad D^\alpha=\frac{\partial^{\alpha_1}}{\partial x_1^{\alpha_1}}\dots\frac{\partial^{\alpha_d}}{\partial x_d^{\alpha_d}},\]
$N$, $N_\alpha\in\mathbb N$; $b^\alpha_j\in\mathbb R$, $x^\alpha_j\in\Omega$, for $j=1,\dots,N^\alpha$; and $\delta_x$ is the Dirac delta function supported at $x$. The problem is as follows: given the potential and its normal derivative on the boundary of the domain, find the locations $x_j^\alpha$ and moments $b_j^\alpha$ of the point sources. The problem has been  extensively studied and many results are available.


The number of point sources in the inverse source problem is often unknown.  In the two-dimensional case studied in~\cite{G10}, the necessary and sufficient conditions for the existence of solution to the ordinary differential equation arising from~\eqref{lin2d} are related to the estimation of the number of dipoles and monopoles counting their multiplicities.  
Below, we extend the estimate to the case of poles of arbitrary degrees for dimension two and make it more precise.

\begin{thm}
Suppose $\Omega$ is a smooth domain in $\mathbb R^2$ and $u$ is a solution of~\eqref{Poisson}-\eqref{f} in $\Omega$, that is,
\begin{equation}\label{potential}
u(x)=u_0(x)+\sum_{|\alpha|\le N}\sum_{j=1}^{N_\alpha}b^\alpha_jD^\alpha\Gamma(x-x^\alpha_j),
\end{equation}
where $u_0$ is a harmonic function on $\bar\Omega$, $\Gamma$ is the fundamental solution of the Laplace equation, and $x_j^\alpha\in\Omega$.  Let $p=|Du|^2$, $q=\partial p/\partial\nu$ and assume that $p>0$ on $\partial\Omega$.  Then,
\begin{equation}\label{noofsources2d}
-\frac1{2\pi}\int_{\partial\Omega}\frac q{2p}\,d\tau\le\sum_{|\alpha|\le N}(|\alpha|+1)N_\alpha,
\end{equation}
Moreover, the equality holds if $p$ does not vanish in $\Omega$.

\end{thm}

\begin{proof}
To obtain estimate~\eqref{noofsources2d} we apply the argument principle of complex analysis
~\cite[p.~152]{Ahlfors}.
Namely, if $f$ is a meromorphic function and $\gamma$ is a simple closed curve that does not pass through zeros or poles of $f$, then
\[\frac1{2\pi i}\int_\gamma\frac{f'}f\,dz=n_+-n_-,\]
where $n_+$ and $n_-$ are the numbers of zeros and poles of $f$, respectively, located inside the region with the boundary $\gamma$.

We apply the argument principle to the function $f=u_y+iu_x$.  First, we discuss how the poles and zeros of $u$ correspond to those of $f$. Since $u$ is harmonic away from its singularities, $f$ is meromorphic and its poles coincide with the poles of $u$.  Moreover, if $u$ has a zero or pole at $z_0=(x_0,y_0)$, then, in polar coordinates $(\rho,\theta)$ at $z_0$ and with $z=x+iy=\rho e^{i\theta}$, we claim that
\begin{itemize}
	\item for $n\neq0$, $u=\mathcal O(\rho^n)$ and $u=o(\rho^{n-1})$ as $\rho\to0$ if and only if $f(z)\sim z^{n-1}$ as $z\to z_0$;
	\item $u=\mathcal O(\log\rho)$ and $u=o(\rho^{-1})$ as $\rho\to0$ if and only if $f\sim z^{-1}$ as $z\to z_0$.
\end{itemize}
To prove the claim in one direction, let $v$ be such harmonic conjugate of $u$ that, for $F=iu-v$, we have $F(z)\sim z^{n}$ for $n\neq0$ or $F(z)\sim\log z$ as $z\to z_0$.  Then $f=F'$ and, therefore, $f(z)\sim z^{n-1}$.  For the other direction, suppose $f(z)=z^{n-1}(c_2+ic_1)+o(z^{n-1})$ for some constants $c_1$ and $c_2$, not both zero.  Then
\begin{equation}\label{Du}
\begin{aligned}
u_x & =\rho^{n-1}(c_1\cos(n-1)\theta+c_2\sin(n-1)\theta)+o(\rho^{n-1}),\\
u_y & =\rho^{n-1}(c_2\cos(n-1)\theta-c_1\sin(n-1)\theta)+o(\rho^{n-1}).
\end{aligned}
\end{equation}
We consider the cases of zeros and poles separately. In the case of a zero at $z_0$, i.\-e., for $n\ge1$, integrating along the line segment connecting $z_0$ and $z$, we have
\[u(z)=u(\rho\cos\theta,\rho\sin\theta)=\int_0^\rho u_x\cos\theta+u_y\sin\theta\,dr.\]
The substitution of~\eqref{Du} into the right-hand side of this representation yields
\begin{equation}\label{u}
u(z)=\frac{\rho^n}n\,(c_1\cos n\theta+c_2\sin n\theta)+o(\rho^n),
\end{equation}
i.\-e., $u=\mathcal O(\rho^n)$ and $u=o(\rho^{n-1})$ as $\rho\to\infty$.  In the case of a pole at $z_0$, i.\-e., for $n\le0$, suppose $f$ has no other poles in the ball of radius $r>0$ centered at $z_0$ so that $\sup_{|z-z_0|=r}u(z)<\infty$.  Integrating along the line segment connecting $z$ and $z_1=z_0+r(z-z_0)/|z-z_0|$, we have
\[u(z)=u(z_1)-\int_\rho^ru_x\cos\theta+u_y\sin\theta\,dr,\]
which, after the substitution of~\eqref{Du} into the right-hand side, becomes~\eqref{u} when $n\le-1$ and $u(z)=c_1\log\rho+o(\log\rho)$ when $n=0$ and the claim is proved. It follows in particular from this claim that $f$ has a pole of order $n+1$, for some $n\ge0$ if and only if $u\sim D^\alpha\Gamma$ with $|\alpha|=n$.

Performing the computation for the argument principle, we have
\begin{align*}
\frac1{2\pi i}\int_{\partial\Omega}\frac{f'}f\,dz & =\frac1{2\pi i}\int_{\gamma}\frac{(u_{xy}+iu_{xx})(u_y-iu_x)}{u_x^2+u_y^2}\,dz\\
& =\frac1{2\pi i}\int_{\partial\Omega}\frac12\left(\frac\partial{\partial x}\log p-i\frac\partial{\partial y}\log p\right)\cdot(dx+i\,dy)\\
& =\frac1{4\pi}\int_{\partial\Omega}\frac\partial{\partial x}\log p\,dy-\frac\partial{\partial y}\log p\,dx\\
& =\frac1{4\pi}\int_{\partial\Omega}\frac\partial{\partial\nu}\log p\,d\tau\\
& =\frac1{2\pi}\int_{\partial\Omega}\frac q{2p}\,d\tau
\end{align*}
Hence
\[-\frac1{2\pi}\int_{\partial\Omega}\frac q{2p}\,d\tau=n_--n_+\le n_-=\sum_{|\alpha|\le N}(|\alpha|+1)N_\alpha,\]
where the last identity follows from the claim above.

\end{proof}


\section{A quasi-linear equation}\label{sec:mielofon}


In this section, we show that harmonic functions satisfy a quasi-linear elliptic equation on a flat portion of the boundary of a domain. 
With this result, we reduce the linear Laplace equation to a non-linear equation but in a lower dimensional space.

To distinguish the Laplace operators in $\mathbb R^n$ and $\mathbb R^{n+1}$ we introduce the following notation:
\[\Delta_ku=\sum_{i=1}^kD_{ii}u\quad\text{and}\quad\Delta u=\Delta_nu,\]
so that the Laplace equation in $\mathbb R^{n+1}$ becomes
\begin{equation}\label{Laplace(n+1)}
\Delta_{n+1}u=\Delta u+D_{n+1,n+1}u=0.
\end{equation}

\begin{thm}
Suppose $U$ is an open set in $\mathbb R^{n+1}$ such that a portion of its boundary $\Omega\subset\{x_{n+1}=0\}\cap\partial U$ is an open set in $\mathbb R^n$. Let $u$ be a harmonic function in $U\cup\Omega$. 
Denote by $Du$ the tangential component of the full gradient of $u$ on $\Omega$, i.e., $Du=(D_1u,\dots,D_nu)$. Let
\begin{equation}\label{p_{n+1}}
p=|Du|^2+(D_{n+1}u)^2\quad\text{and}\quad q=D_{n+1}p\,.
\end{equation}
Then $u$ satisfies
\begin{equation}\label{mielofon}
\div\frac{Du}{\sqrt{p-|Du|^2}}+\frac12\frac{\sigma q}{p-|Du|^2}=0
\end{equation}
in $\Omega$, where $\div$\-is the divergence operator in $\mathbb R^n$ and $\sigma=\sgn D_{n+1}u$. 
\end{thm}
\begin{proof}
Taking partial derivatives of the first equation in~\eqref{p_{n+1}}, we obtain
\begin{equation}\label{D_ip}
2D_{n+1}uD_{i,n+1}u=D_ip-2D_juD_{ij}u,\quad i=1,\dots,n,
\end{equation}
\[2D_iuD_{i,n+1}u+2D_{n+1}uD_{n+1,n+1}u=D_{n+1}p=q,\]
where, here and throughout, the summation from 1 to $n$ is assumed for repeated indices $i$ and $j$. Multiplying the last equation by $D_{n+1}u$ and replacing in it the terms involving the partial derivative with respect to $x_{n+1}$ using~\eqref{D_ip} and~\eqref{Laplace(n+1)} we arrive at
\[D_iu(D_ip-2D_juD_{ij}u)+2(p-(D_iu)^2)(-\Delta u)=qD_{n+1}u.\]
Rearranging the terms further, we obtain
\begin{equation}\label{non-div}
(p-|Du|^2)\Delta u+D_iuD_juD_{ij}u-\frac12D_ipD_iu+\frac q2\sigma\sqrt{p-|Du|^2}=0,
\end{equation}
the equation involving the partial derivatives with respect to $x_1,\dots,x_n$ only.  Finally, equation~\eqref{mielofon} is the divergence form of equation~\eqref{non-div}.
\end{proof}
Equation~\eqref{mielofon} is elliptic but not uniformly elliptic: the eigenvalues are $\lambda=p-|Du|^2$ and $\Lambda=p$. In particular, the ellipticity is lost where the full gradient of $u$ is tangent to the boundary, i.e., on the magnetic equator as referred to in~\cite{Cherkaeva}.

The equation has an outward resemblance to 
the minimal surface equation
\[(1+|Du|^2)\Delta u-D_iuD_juD_{ij}u=0\] 
in the divergence part of the operator.
Also equation~\eqref{mielofon} is similar to the equation of gas dynamics
\[\Delta u-\frac{D_i uD_j u}{1-\dfrac{\gamma-1}2|Du|^2}D_{ij}u=0,\]
in that the former loses ellipticity when $|Du|^2=p$ and the latter changes the type from elliptic when $|Du|<[2/(\gamma+1)]^{1/2}$ to hyperbolic when $[2/(\gamma+1)]^{1/2}<|Du|<[2/(\gamma-1)]^{1/2}$.

\section{Numerical studies}\label{sec:numer}
In our numerical experiments, we consider equation~\eqref{mielofon}
on the unit square in $\mathbb R^2$.  
The data is generated from functions of the form~\eqref{potential} that are harmonic in $\mathbb R^3$ with the exception of singularities, if any. Specifically, we choose the function $u$ from among the following choices:
\begin{itemize}
\item $u_0(x,y,z)=(x-x_0)^2+(y-y_0)^2-2(z-z_0)^2$, where
$(x_0,y_0,z_0)=(-2,-3,-2.5)$;
\item $u_1(x,y,z)=\Gamma(x-x_1,y-y_1,z-z_1)$, where
\[\Gamma(x,y,z)=-\frac1{4\pi}\frac1{\sqrt{x^2+y^2+z^2}}\]
is the fundamental solution of the Laplace equation in $\mathbb R^3$ and $(x_1,y_1,z_1)=(0.2,0.1,0.5)$;
\item $u_0+u_1$.
\end{itemize}
The function $u_0$ plays the role of a potential for the background field; the function $u_1$ is a monopole and it models a small perturbation; and $u_0+u_1$ generates the combined field.

The values of the data: $p$, $q$, $\sigma$, and of the exact solution (for boundary values and error estimates) are collected on the unit square embedded in the coordinate $xy$-plane, i.e., $\Omega=[0,1]\times[0,1]\times\{z=0\}$. We supplement equation~\eqref{mielofon} with the Dirichlet boundary conditions. Although this choice of data requires the knowledge of $u$ on $\partial\Omega$ and it is not realistic in practice, we use it to 
take advantage
of the available boundary value solver.
In this sense, the function $\sigma$ is another exogenous assumption on the data. In all of our examples, $\partial u/\partial\nu$ is sign definite, i.e., $\sigma$ is constant.

We plot the function $u$ and corresponding $|Du|^2$ and $(\partial/\partial\nu)|Du|^2$ in Figure~\ref{Fig:upq}. 
We observe that the range of values of both the potential $u$ and the magnitude of the gradient $|Du|^2$ in the perturbation is two orders of magnitude smaller than those in the background field. 
The plots of these functions in the background field and in the combined field are virtually indistinguishable by visual inspection.
The effect of the perturbation is more pronounced however for 
$(\partial/\partial\nu)|Du|^2$. This function is constant in the background field and has a local extremum near the origin in the perturbation. The combined field inherits the magnitude from the background field and the shape from the perturbation.

\begin{figure}[t]
\captionsetup[subfigure]{labelformat=empty}
	\centering
	\begin{subfigure}[t]{0.3\textwidth}
		\includegraphics[height=1.25in]{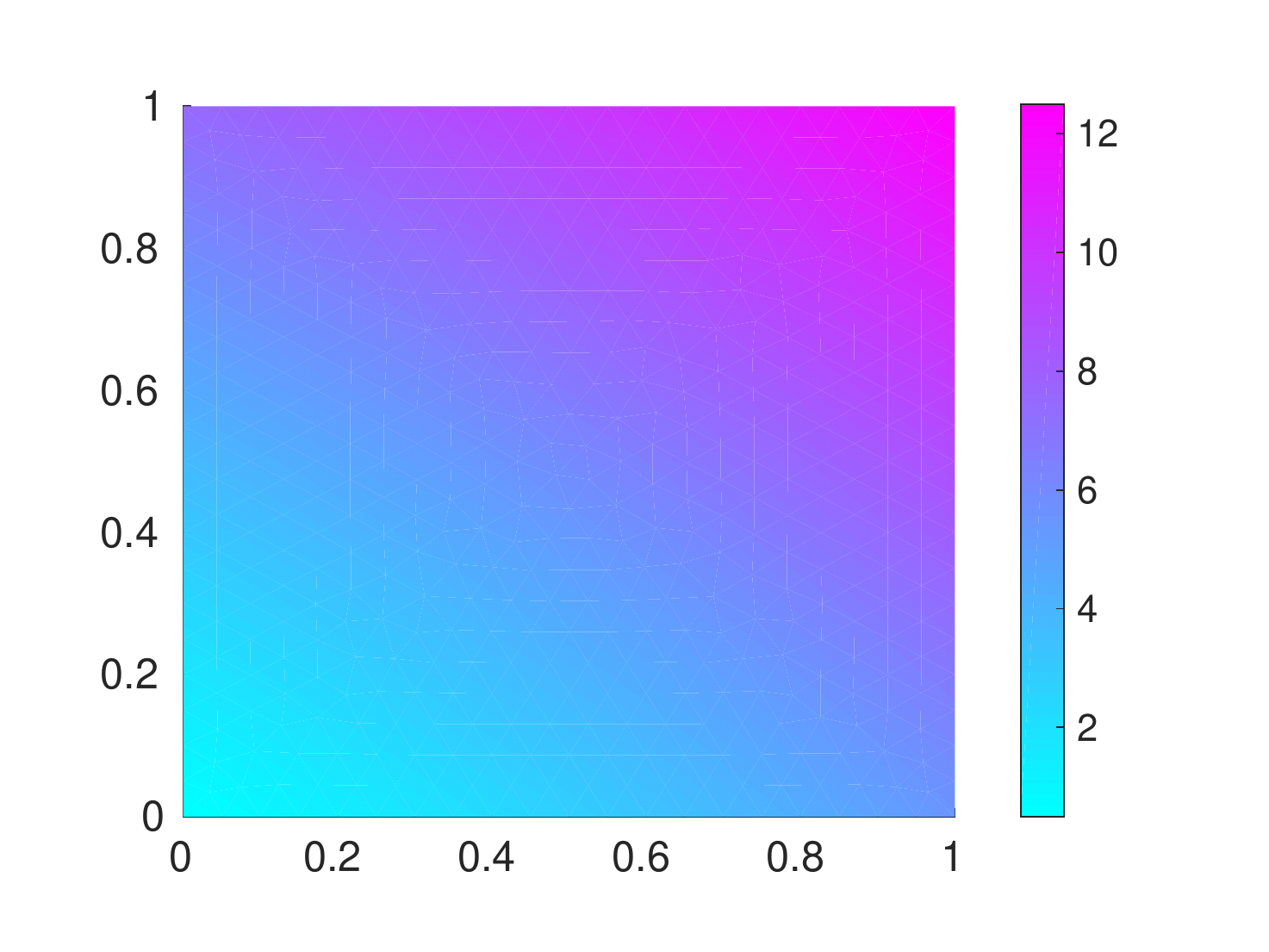}
		\caption{$u_0$}
	\end{subfigure}
	~ 
	\begin{subfigure}[t]{0.3\textwidth}
	\centering
		\includegraphics[height=1.25in]{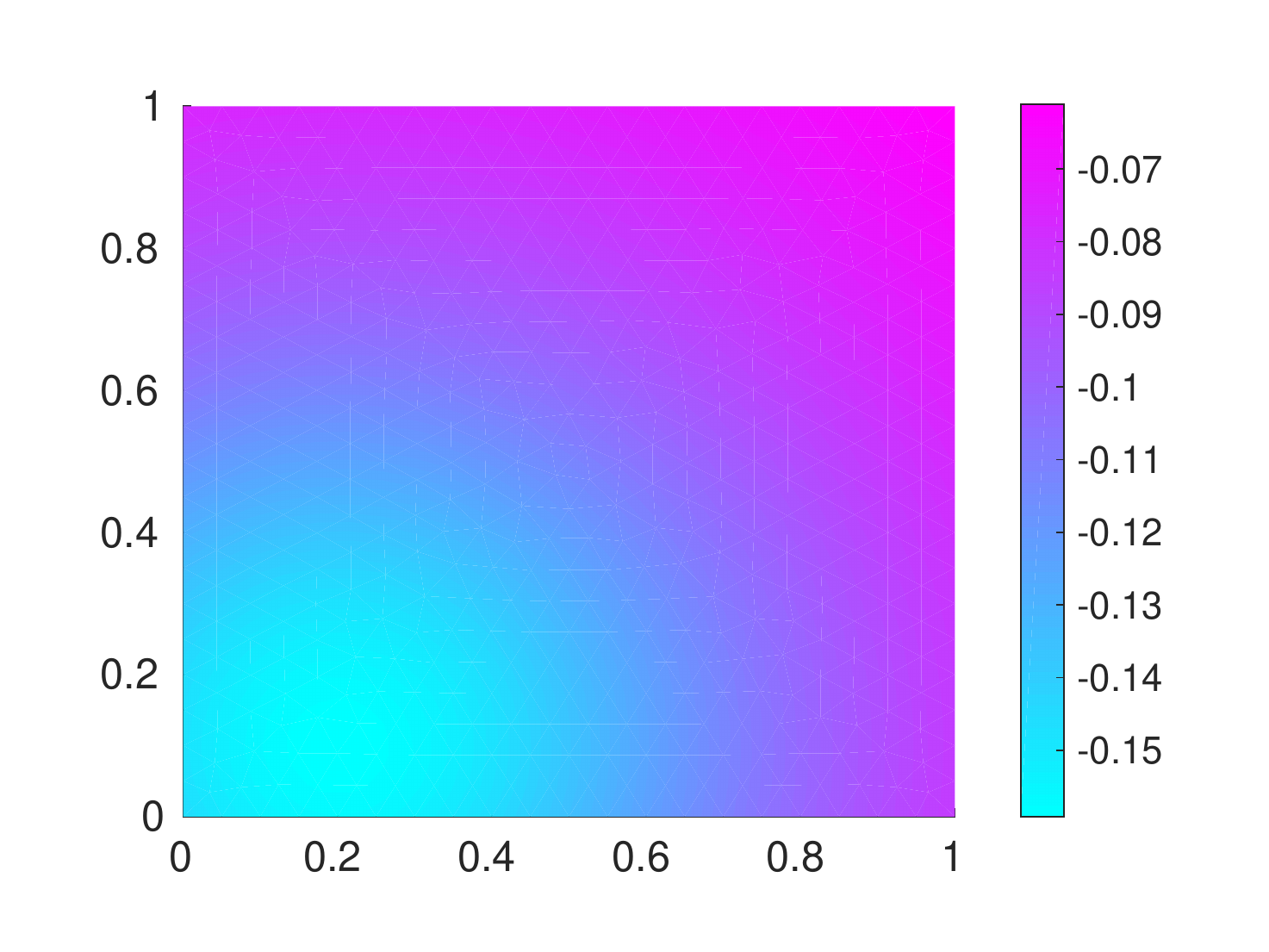}
		\caption{$u_1$}
	\end{subfigure}
	~ 
	\begin{subfigure}[t]{0.3\textwidth}
	\centering
		\includegraphics[height=1.25in]{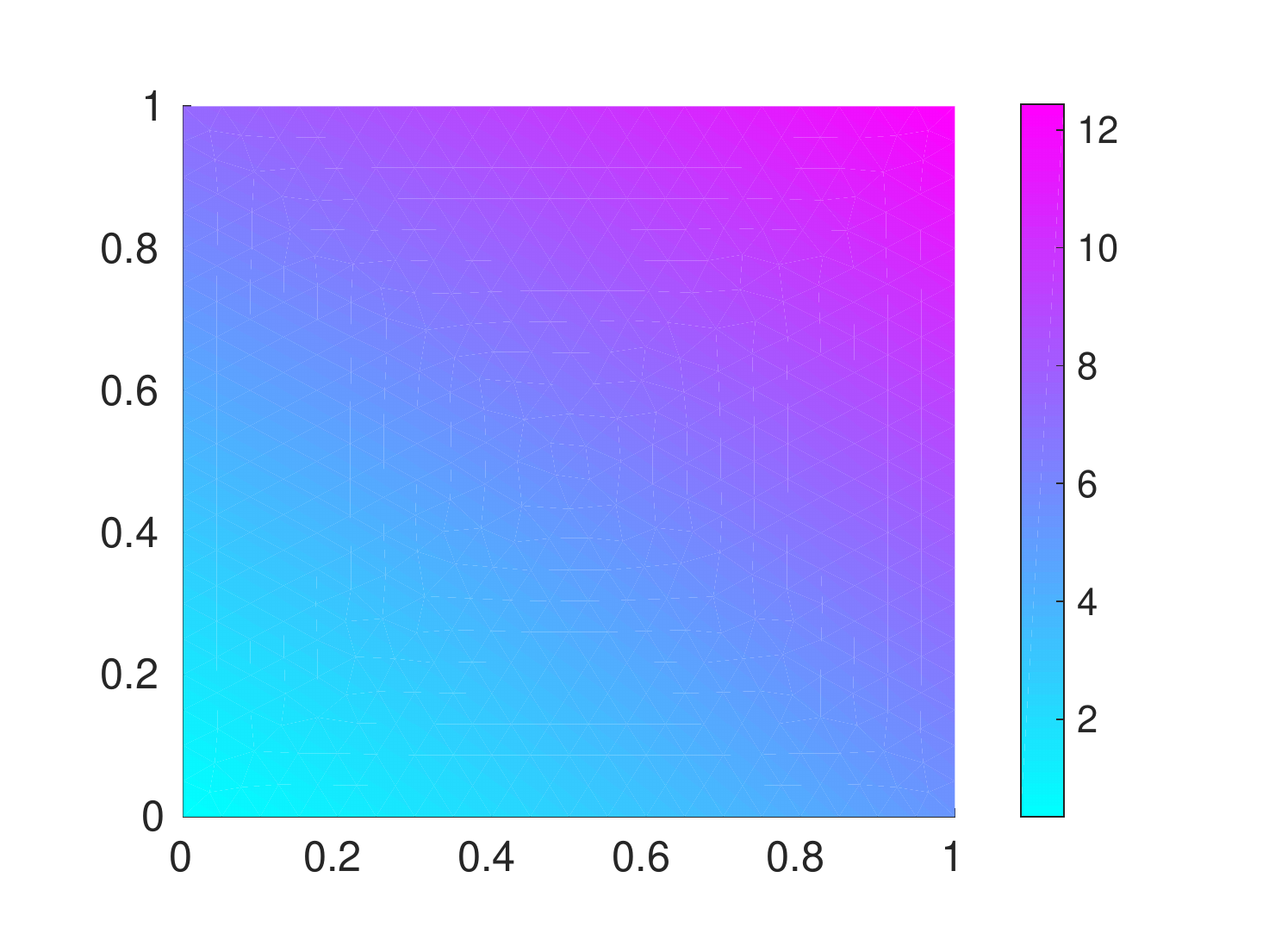}
		\caption{$u_0+u_1$}
	\end{subfigure}\\
	\begin{subfigure}[t]{0.3\textwidth}
		\includegraphics[height=1.25in]{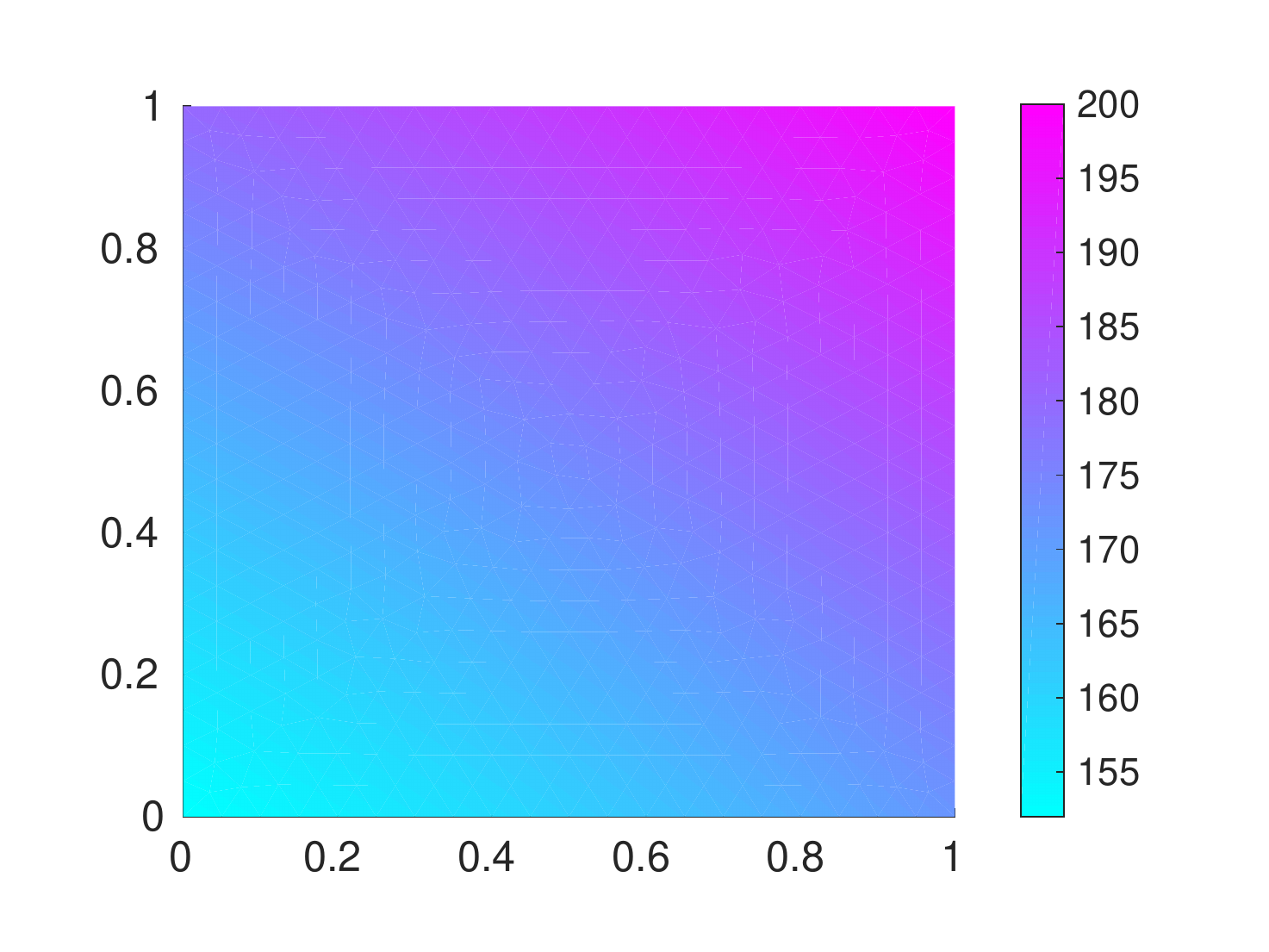}
		\caption{$|Du_0|^2$}
	\end{subfigure}
	~ 
	\begin{subfigure}[t]{0.3\textwidth}
	\centering
		\includegraphics[height=1.25in]{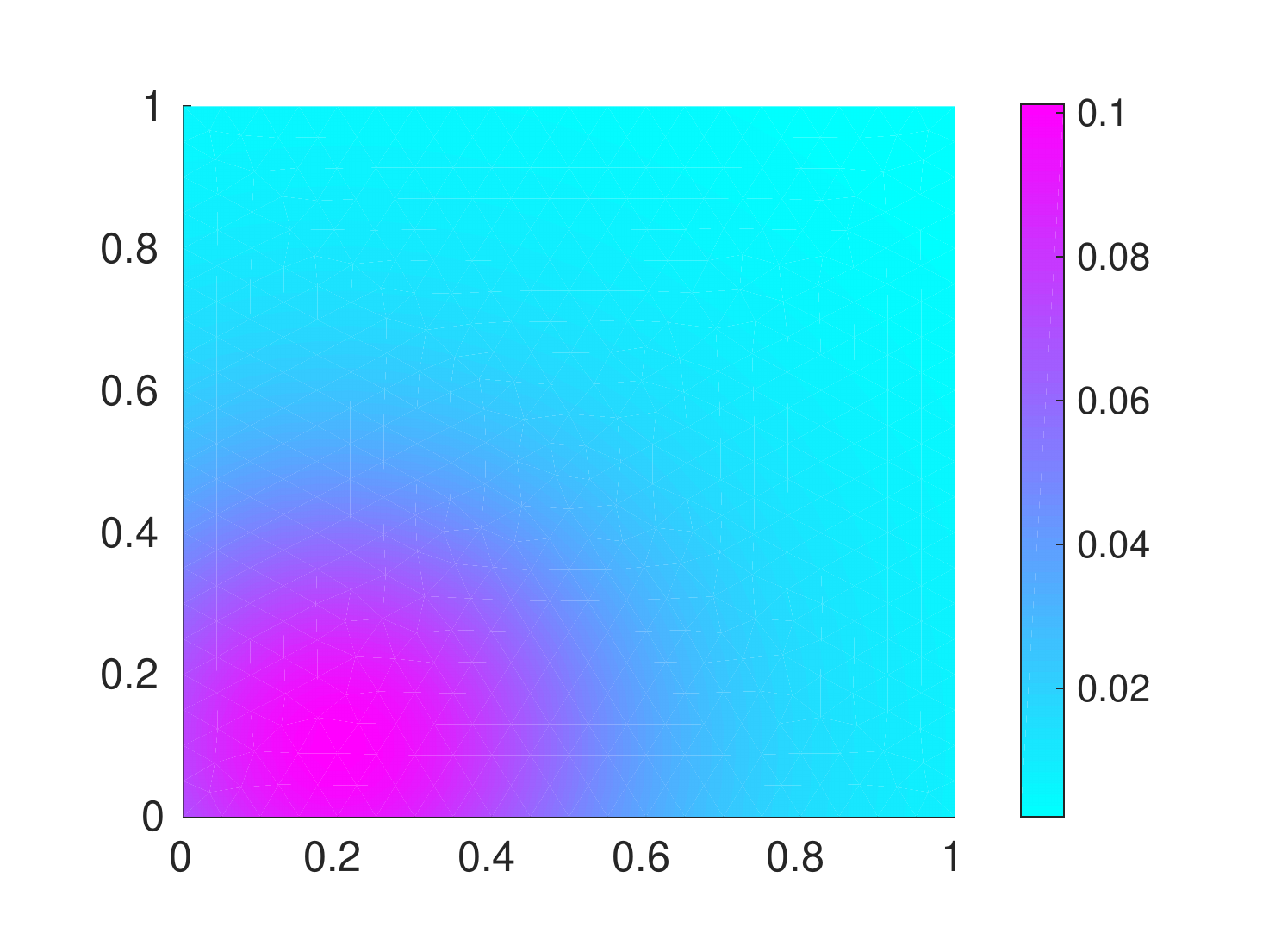}
		\caption{$|Du_1|^2$}
	\end{subfigure}
	~ 
	\begin{subfigure}[t]{0.3\textwidth}
	\centering
		\includegraphics[height=1.25in]{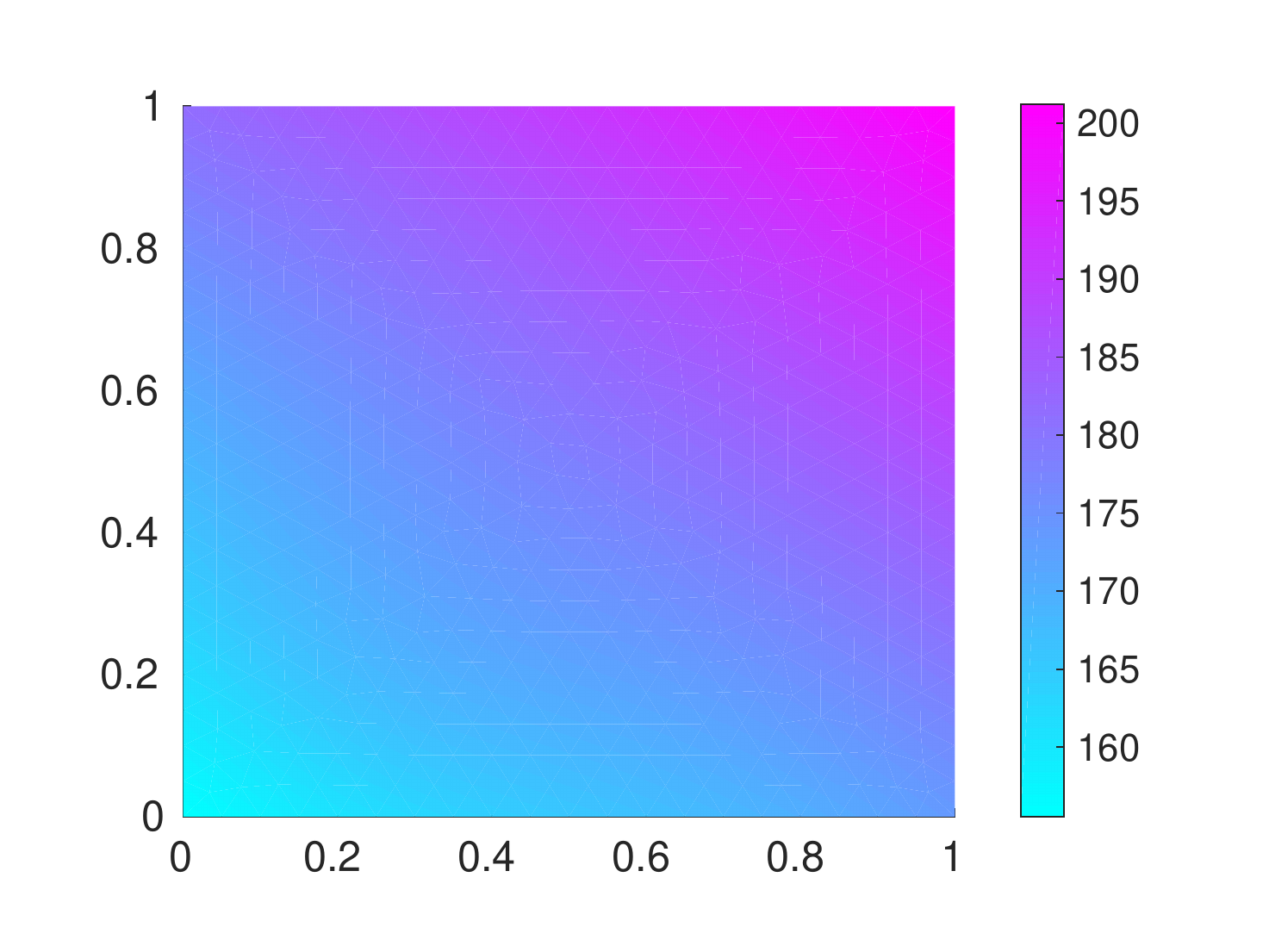}
		\caption{$|D(u_0+u_1)|^2$}
	\end{subfigure}\\
	\centering
	\begin{subfigure}[t]{0.3\textwidth}
		\includegraphics[height=1.25in]{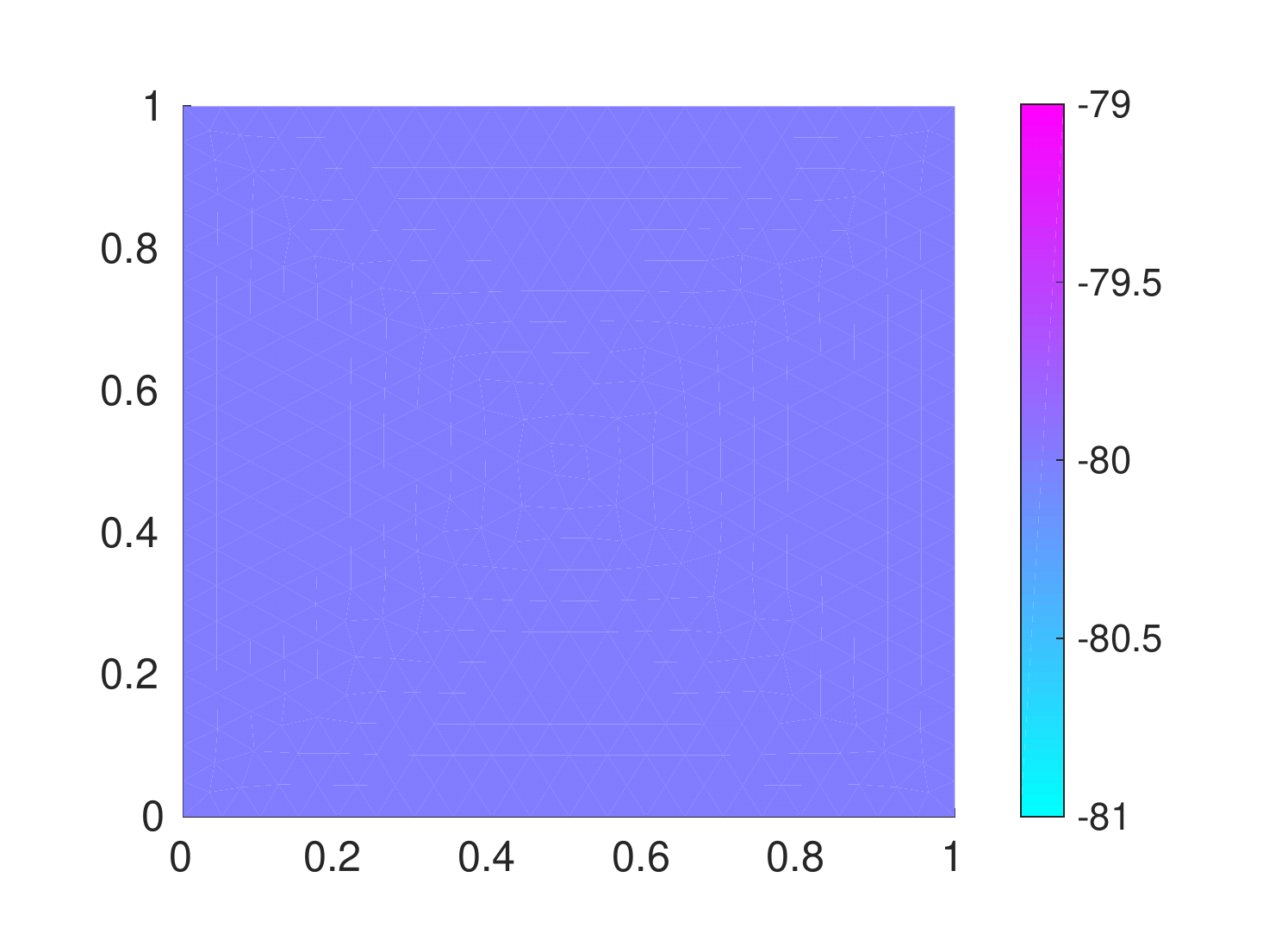}
		\caption{$(\partial/\partial\nu)|Du_0|^2$}
	\end{subfigure}
	~ 
	\begin{subfigure}[t]{0.3\textwidth}
	\centering
		\includegraphics[height=1.25in]{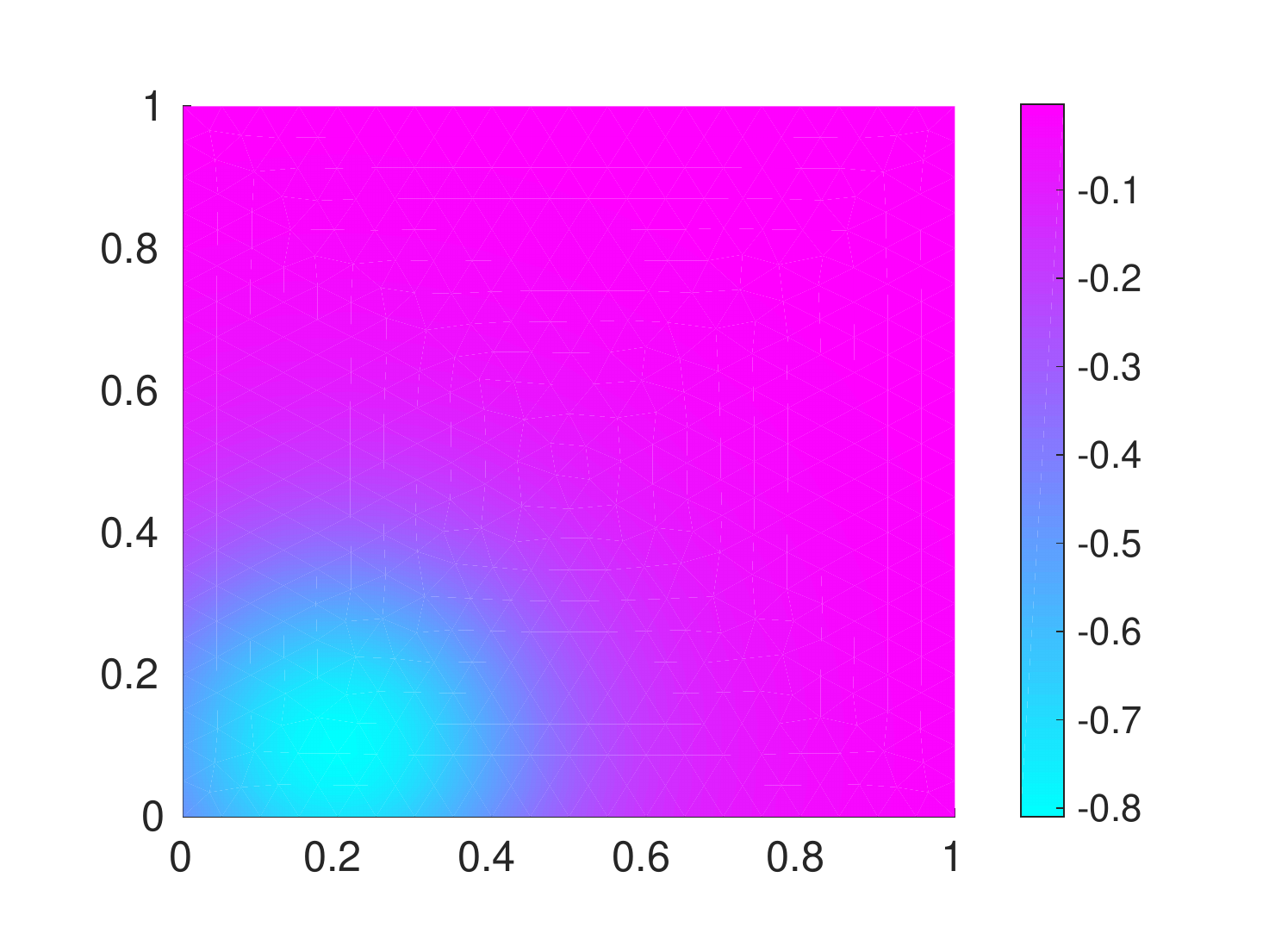}
		\caption{$(\partial/\partial\nu)|Du_1|^2$}
	\end{subfigure}
	~ 
	\begin{subfigure}[t]{0.3\textwidth}
	\centering
		\includegraphics[height=1.25in]{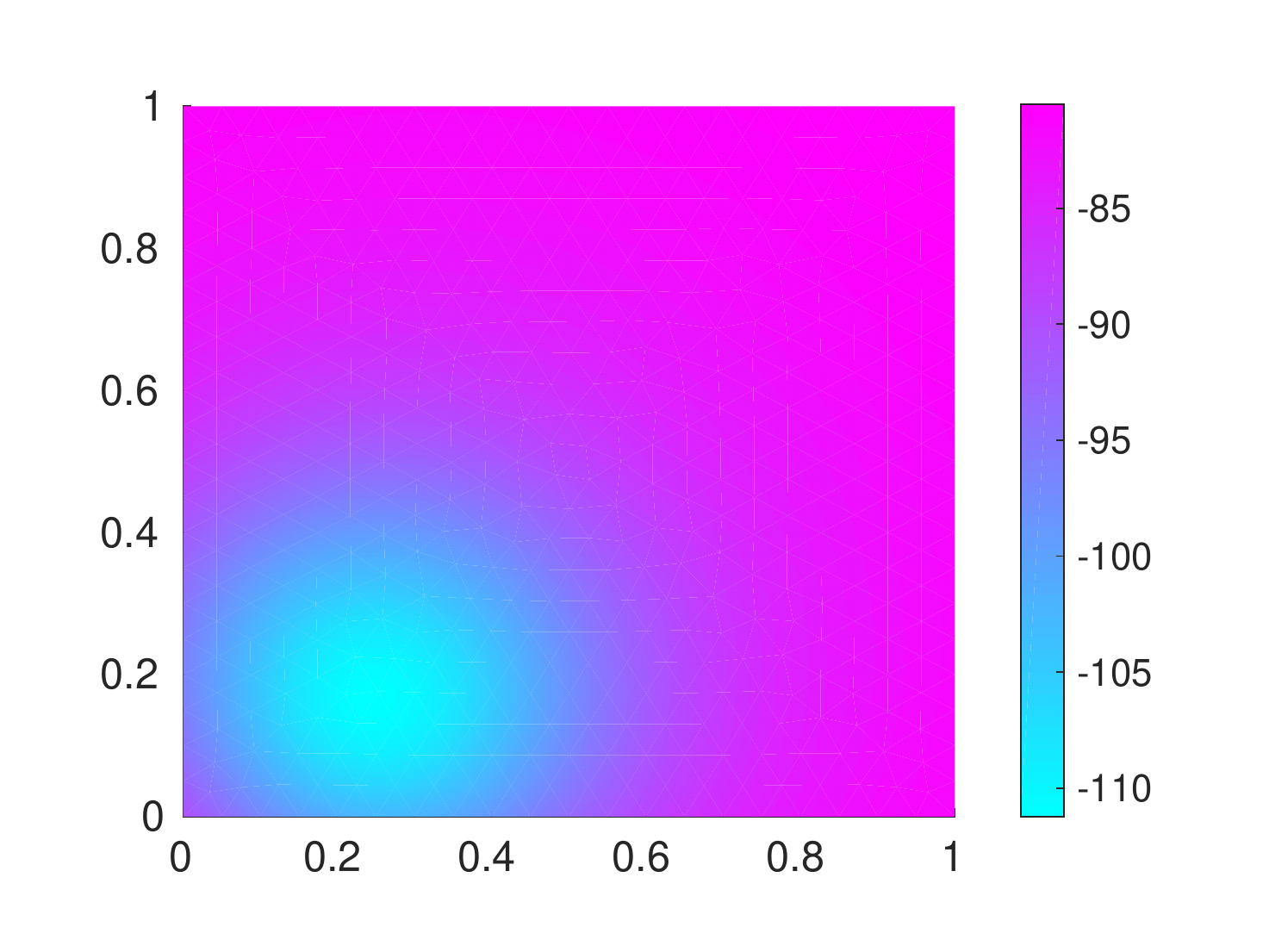}
		\caption{$(\partial/\partial\nu)|D(u_0+u_1)|^2$}
	\end{subfigure}
	\caption{Plots of the potentials and the corresponding augmented data}
	\label{Fig:upq}
\end{figure}

We apply the finite element method as it is implemented in the PDE Toolbox of Matlab. Being in divergence form, equation ~\eqref{mielofon} admits a weak formulation:
\[\int_\Omega\frac{Du\cdot Dv}{\sqrt{p-|Du|^2}}+\frac12\frac{\sigma qv}{p-|Du|^2}\,dx=0\quad\text{for }v\in C^1_0(\Omega).\]
This equation is discretized using the linear elements. We compute the solution for the values of the target maximum mesh edge length of the triangular elements $H_{max}$ ranging from 0.025 to 0.2\,. The stock solver in the PDE Toolbox uses the Gauss-Newton iteration which requires an initial guess. On the coarsest mesh, we solve the two-dimensional Laplace equation with the same Dirichlet boundary values as in the non-linear problem and supply the resulting harmonic function as an initial guess to the non-linear solver. For successive mesh refinements, the initial guess is an interpolation of the solution from the previous mesh. The estimates of the $L^2$- and $H^1$-norms of the errors are based on the comparison with the exact solution.  They are summarized along with the estimated rates of convergence in Table~\ref{errRates}. The log-log plots of these errors appear in Figure~\ref{errorPlots}. 
\begin{center}
\begin{table}[thb]
\begin{tabular}{c|c|c|c|c|c|c}
& \multicolumn{2}{|c|}{$u_0$} & \multicolumn{2}{|c|}{$u_1$} & \multicolumn{2}{|c}{$u_0+u_1$}\\\hline
$H_{max}$ & $L^2$ & $H^1$ & $L^2$ & $H^1$ & $L^2$ & $H^1$\\\hline
0.2 & 1.58e-03 & 2.04e-02 & 3.00e-04 & 2.08e-03 & 1.56e-03 & 2.17e-02\\
0.1 & 2.84e-04 & 7.90e-03 & 8.59e-05 & 1.01e-03 & 2.86e-04 & 8.42e-03\\
0.05 & 6.40e-05 & 3.50e-03 & 2.13e-05 & 4.05e-04 & 6.63e-05 & 3.73e-03\\
0.025 & 1.12e-05 & 1.23e-03 & 5.31e-06 & 1.53e-04 & 1.29e-05 & 1.32e-03\\\hline
Rate & 2.3562 & 1.3302 & 1.9470 & 1.2602 & 2.2858 & 1.3290
\end{tabular}
\caption{Estimated errors and rates of convergence}
\label{errRates}
\end{table}
\end{center}
\begin{figure}[bht]
\includegraphics[width=0.8\textwidth]{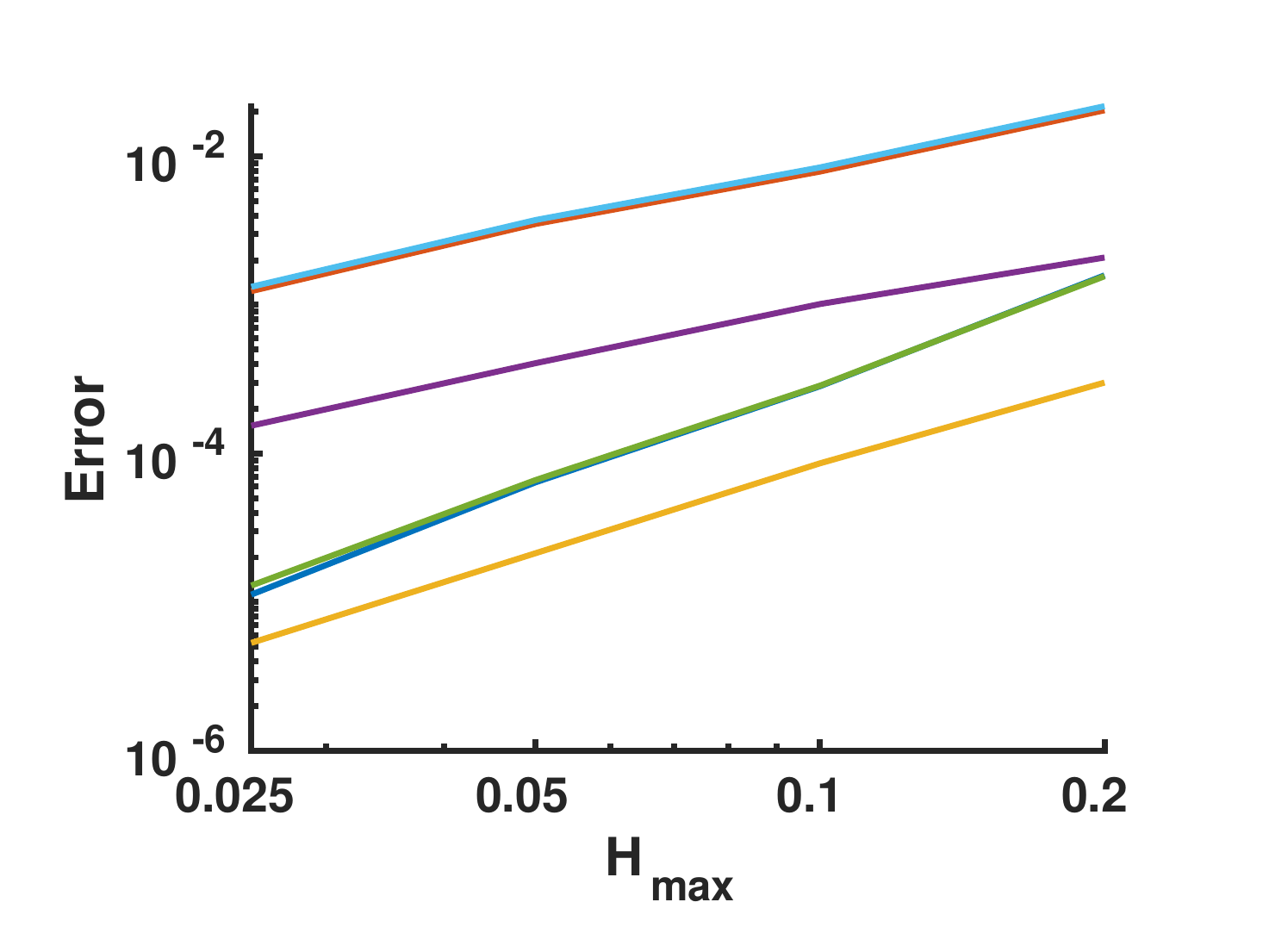}
\caption{Log-log plots of errors}
\label{errorPlots}
\end{figure}

\section*{Conclusions}
The Backus problem is not well-posed and requires extra conditions to guarantee uniqueness. We consider this problem with additional data that is in the same paradigm as the non-linear boundary values in the original formulation of the problem. We show how the expanded data may be used to estimate the number of sources in the corresponding inverse source problem on the plane. In addition, we derive a non-linear equation in terms of the augmented data that is satisfied by the harmonic function on a manifold of lower dimension. We demonstrate that the resulting equation is amenable to numerical solution by standard methods and we recover some solutions using commonly available software.
\bibliography{bibliography.bib}
\end{document}